\documentclass{amsart}
\usepackage{cite, colonequals, comment, enumitem, mathtools, xcolor}
\usepackage[hidelinks]{hyperref}

\author{Masood Aryapoor}
\address[Masood Aryapoor]{Division of Mathematics and Physics, M\"alar\-dalen  University,  Box 325, SE-631 05 Eskilstuna, Sweden}
\email{masood.aryapoor@mdu.se}

\author{Per B\"ack}
\address[Per B\"ack]{Division of Mathematics and Physics, M\"alar\-dalen  University,  Box  883,  SE-721  23  V\"aster\r{a}s, Sweden}
\email[corresponding author]{per.back@mdu.se}

\date\today

\newtheorem{corollary}{Corollary}
\newtheorem{lemma}{Lemma}
\newtheorem{proposition}{Proposition}
\newtheorem{theorem}{Theorem}

\theoremstyle{definition}
\newtheorem{definition}{Definition}
\newtheorem{example}{Example}

\theoremstyle{remark}
\newtheorem{remark}{Remark}

\DeclareMathOperator{\Cay}{Cay}

\DeclareMathOperator{\id}{id}
\DeclareMathOperator{\im}{im}
\DeclareMathOperator{\fl}{f\mkern 1mu l}

\begin{document}

\begin{abstract}We introduce and study flipped non-associative polynomial rings. In particular, we show that all Cayley--Dickson algebras naturally appear as quotients of a certain type of such rings; this extends the classical construction of the complex numbers (and quaternions) as a quotient of a (skew) polynomial ring to the octonions, and beyond. We also extend some classical results on algebraic properties of Cayley--Dickson algebras by McCrimmon to a class of flipped non-associative polynomial rings. 
\end{abstract}

\title[Flipped non-associative polynomial rings]{Flipped non-associative polynomial rings and the Cayley--Dickson construction}

\maketitle

\section{Introduction}
The Cayley--Dickson construction, introduced by Dickson \cite{Dic19} based on previous work by Cayley \cite{Cay45} and then generalized by Albert \cite{Alb42}, is a famous construction for generating new $*$-algebras out of old ones. It is perhaps best known for generating all the real, normed division algebras: the real numbers ($\mathbb{R}$), the complex numbers ($\mathbb{C}$), the quaternions ($\mathbb{H}$), and the octonions ($\mathbb{O}$) \cite{Alb47, UW60, Wri53}. However, the construction is undoubtedly quite mysterious; a first (second and third\ldots) encounter with its strange product leaves one rather puzzled and with a feeling that it is a patchwork created ad hoc to make certain algebras, like the above, fit in a construction. In this article, we try to shed new light on the Cayley--Dickson construction with the purpose of illuminating the underlying patchwork. To this end, we introduce and study the notion of \emph{flipped non-associative polynomial rings}. In particular, we show that all Cayley--Dickson algebras naturally appear as quotients of a certain type of such rings (\autoref{thm:cayley-dickson-isomorphism}). By using this new class of polynomial rings, we may in particular extend the classical construction of $\mathbb{C}$ (and $\mathbb{H}$) as a quotient of a (skew) polynomial ring to $\mathbb{O}$, and beyond (see \autoref{ex:classical-Cayley-Dickson-quotient}).

The flipped non-associative polynomial rings that give rise to the Cayley--Dickson algebras are ``flipped'' versions of a type of non-associative and non-commutative polynomial rings known as \emph{non-associative Ore extensions}, introduced by Nystedt, Öinert, and Richter \cite{NOR18}. Non-associative Ore extensions are in turn non-associative generalizations of \emph{Ore extensions}, the latter introduced by Ore \cite{Ore33} under the name ``non-commutative polynomial rings''. Since their introduction, Ore extensions have been studied quite extensively (see e.g. \cite{GW04, Lam01, MR01} for comprehensive introductions). 

We study the underlying non-associative and non-commutative ring structure of flipped non-associative Ore extensions and discover, rather surprisingly, that the flipped non-associative Ore extensions that give rise to the Cayley--Dickson algebras are themselves Cayley--Dickson algebras of a certain kind (\autoref{thm:cayley-dickson-isomorphism1}). With this discovery, we may rather easily prove results regarding their algebraic properties (\autoref{thm:Cayley-Dickson-properties} and \autoref{thm:cayley-dickson-centers}). These results are in turn natural counterparts of some classical results on Cayley--Dickson algebras obtained by McCrimmon \cite[Theorem 6.8 (i)--(v), (viii)--(xii)]{McC85}.

The article is organized as follows:

In \autoref{sec:preliminaries}, we provide conventions and preliminaries from non-associative ring theory (\autoref{subsec:non-assoc-ring-theory}). We also recall what Ore extensions are (\autoref{subsec:ore-extensions}) and how the definition of these rings can naturally be extended to the non-associative setting (\autoref{subsec:non-assoc-ore}).

In \autoref{sec:flipped-non-associative-ore-extensions}, we introduce the concept of flipped non-associative polynomial rings (\autoref{def:flipped-ring}) and flipped non-associative Ore extensions (\autoref{def:flip-ore}). We then prove results regarding their underlying non-associative and non-commutative ring structure.

In \autoref{sec:cayley-dickson}, we give a brief introduction to $*$-algebras and the Cayley--Dickson construction. We then show that all Cayley--Dickson algebras naturally appear as quotients of certain flipped non-associative Ore extensions (\autoref{thm:cayley-dickson-isomorphism}), and that these flipped non-associative Ore extensions are in turn Cayley--Dickson algebras of a particular type (\autoref{thm:cayley-dickson-isomorphism1}). We conclude the article by proving results regarding algebraic properties of the aforementioned flipped non-associative Ore extensions (\autoref{thm:Cayley-Dickson-properties} and \autoref{thm:cayley-dickson-centers}). 

\section{Preliminaries}\label{sec:preliminaries}
\subsection{Non-associative ring theory}\label{subsec:non-assoc-ring-theory}We denote by $\mathbb{N}$ the set of natural numbers. All rings and algebras are assumed to be unital, unless stated otherwise, and the multiplicative identity element is written $1$. Any endomorphism is assumed to respect $1$.  By a \emph{non-associative ring}, we mean a ring which is not necessarily associative. If $R$ is a non-associative ring, by a \emph{left $R$-module}, we mean an additive group $M$ equipped with a biadditive map $R\times M\to M$, $(r,m)\mapsto rm$ for any $r\in R$ and $m\in M$. A subset $B$ of $M$ is a basis if any $m\in M$ can be uniquely written as a sum $m=\sum_{b\in B}r_bb$ with finite support, where $r_b\in R$. A left $R$-module that has a basis is called \emph{free}.

We also recall the following standard concepts and notations from non-associative ring theory (see e.g. \cite[Section 2]{NOR18}). Let $R$ be a non-associative ring. For  $r,s,t\in R$, we set $[r,s]\colonequals rs-sr$ and $(r,s,t)\colonequals (rs)t-r(st)$. We define the \emph{left}, \emph{middle}, and \emph{right nucleus} of $R$ as follows: $N_l(R)\colonequals\{r\in R\colon (r,s,t)=0 \text{ for all } s,t\in R\}$, $N_m(R)\colonequals\{s\in R\colon (r,s,t)=0 \text{ for all } r,t\in R\}$, and $N_r(R)\colonequals\{t\in R\colon (r,s,t)=0 \text{ for all } r,s\in R\}$. It turns out that $N_l(R)$, $N_m(R)$, and $N_r(R)$ are all associative subrings of $R$. We also define the \emph{nucleus} of $R$, denoted by $N(R)$, as the set $N_l(R)\cap N_m(R)\cap N_r(R)$. The \emph{center} of $R$, $Z(R)$, is defined as  $C(R)\cap N(R)$, where $C(R)\colonequals\{r\in R\colon [r,s]=0 \text{ for all } s\in R\}$ is the \emph{commuter} of $R$. In particular, $Z(R)$ is an associative and commutative subring of $R$. 

A non-associative ring $D$ is called a \emph{division ring} if the equations $r=qs$ and $r=sq'$, where $r,s\in D$ with $s\neq0$, have unique solutions $q,q'\in R$. Since we require our rings to be unital, for each non-zero $s\in D$, there must be unique $q,q'\in D$ such that $1=qs$ and $1=sq'$ hold ($q$ need not equal $q'$, however). Recall that over the real numbers, there are precisely four normed division algebras: $\mathbb{R}$, $\mathbb{C}$, $\mathbb{H}$, and $\mathbb{O}$. In particular, $\mathbb{H}$ is not commutative and $\mathbb{O}$ is neither commutative, nor associative.

\subsection{Ore extensions}\label{subsec:ore-extensions}
We begin by recalling the definition of Ore extensions as given in \cite[Definition 1.1]{NOR18}.

\begin{definition}[Ore extension]
Let $S$ be a ring, and let $R$ be an associative subring of $S$ containing the multiplicative identity element $1$ of $S$. Then $S$ is called an \emph{Ore extension of $R$} if there is an element $x\in S$ such that the following axioms hold:
\begin{enumerate}[label=\upshape(O\arabic*)]
\item $S$ is associative;\label{it:o3}
\item $S$ is a free left $R$-module with basis $\{1,x,x^2,\ldots\}$;\label{it:o1}
\item $xR\subseteq Rx+R$.\label{it:o2}
\end{enumerate}
\end{definition}
 As usual, we use the convention that $x^0\colonequals 1$. If \ref{it:o2} is replaced by
\begin{enumerate}[label=\upshape(O\arabic*)$'$]\addtocounter{enumi}{2}
	\item $xR\subseteq Rx$;
\end{enumerate}
then $S$ is called a \emph{skew polynomial ring}.

If $R$ is a non-associative ring with an endomorphism $\sigma$, then an additive map $\delta$ on $R$ is called a \emph{left $\sigma$-derivation} if for all $r,s\in R$,
\[
\delta(rs)=\sigma(r)\delta(s)+\delta(r)s.	
\]
Similarly, an additive map $\delta'$ on $R$ is called a \emph{right $\sigma$-derivation} if for all $r,s\in R$,
\[
\delta'(rs)=\delta'(r)\sigma(s)+r\delta'(s).	
\]

 An \emph{ordinary generalized polynomial ring} $R[X;\sigma,\delta]$ (see e.g. \cite[Section 1]{NOR18}) where $\sigma$ is an endomorphism and $\delta$ is a left $\sigma$-derivation on $R$ consists of left polynomials in $X$ with coefficients in $R$. The addition in this ring is the ordinary addition of polynomials, and the multiplication is defined by the biadditive extension of the relations
\begin{equation}
(rX^m)(sX^n)=\sum_{i\in\mathbb{N}}(r\pi_i^m(s))X^{i+n},\label{eq:prod-ore}	
\end{equation}
where the functions $\pi_i^m\colon R\to R$, called \emph{$\pi$ functions}, are defined as the sum of all $\binom{m}{i}$ compositions of $\sigma$ and $\delta$ in which $\sigma $ occurs $i$ times and  $\delta$ occurs $m-i$ times. For instance, $\pi_2^3=\sigma\circ\sigma\circ\delta + \sigma\circ\delta\circ\sigma + \delta\circ\sigma\circ\sigma$, while $\pi_0^0$ is defined as $\id_R$. Whenever $i>m$ or $i<0$, we set $\pi_i^m=0$. It is known that $R[X;\sigma,\delta]$ is an Ore extension of $R$ with $x=X$ (for a proof, see e.g. \cite[Proposition 3.2]{NOR18}). If $\delta=0$, then $R[X;\sigma,\delta]$ is a skew polynomial ring, and the product \eqref{eq:prod-ore} is 
 \begin{equation}
 	(rX^m)(sX^n)=(r\sigma^m(s))X^{m+n}.\label{eq:skew-multiplication}
 \end{equation}
 If $\sigma=\id_R$ and $\delta=0$, then $R[X;\sigma,\delta]$ is the ordinary polynomial ring $R[X]$. The above construction moreover gives us all Ore extensions of $R$ in the sense that any Ore extension of $R$ is isomorphic to an ordinary generalized polynomial ring $R[X;\sigma,\delta]$ (for a proof, see e.g. \cite[Proposition 3.3]{NOR18}).

\subsection{Non-associative Ore extensions}\label{subsec:non-assoc-ore}
In \cite{NOR18}, the authors noted that the product \eqref{eq:prod-ore} equips the additive group $R[X;\sigma,\delta]$ of ordinary generalized polynomials over any  non-associative ring $R$ with a non-associative ring structure for any two additive maps $\sigma$ and $\delta$ on $R$ satisfying $\sigma(1)=1$ and $\delta(1)=0$. Hence, in order to define non-associative Ore extensions, they investigated how to adapt the axioms \ref{it:o3}, \ref{it:o1}, and \ref{it:o2} so that these rings would still correspond to the above generalized polynomial rings. They then suggested the following definition:

\begin{definition}[Non-associative Ore extension]\label{def:non-assoc-ore}
Let $S$ be a non-associative ring, and let $R$ be a subring of $S$ containing the multiplicative identity element $1$ of $S$. Then $S$ is called a \emph{non-associative Ore extension of $R$} if there is an element $x\in S$ such that the following axioms hold:
\begin{enumerate}[label=\upshape(N\arabic*)]
\item  $x\in N_r(S)\cap N_m(S)$;\label{it:no3}
\item $S$ is a free left $R$-module with basis $\{1,x,x^2,\ldots\}$;\label{it:no1}
\item $xR\subseteq Rx+R$.\label{it:no2}
\end{enumerate}
\end{definition}
If \ref{it:no2} is replaced by
\begin{enumerate}[label=\upshape(N\arabic*)$'$]\addtocounter{enumi}{2}
	\item $xR\subseteq Rx$;
\end{enumerate}
then $S$ is called a \emph{non-associative skew polynomial ring}.

\begin{remark}
Axiom \ref{it:no3} ensures that the element $x$ associates with itself, so that there is no ambiguity in writing $x^n$ for the $n$-fold product of $x$ with itself, for $n\in\mathbb{N}$. 
\end{remark}

Let $R$ be a non-associative ring. We denote by $R[X]$ the set of formal sums $\sum_{i\in \mathbb{N}}r_iX^i$ of finite support, where $r_i\in R$, equipped with pointwise addition. Now, let $\sigma$ and $\delta$ be additive maps on $R$ satisfying $\sigma(1)=1$ and $\delta(1)=0$. The \emph{generalized polynomial ring} $R[X;\sigma,\delta]$ over $R$ is defined as the additive group $R[X]$ with multiplication defined by \eqref{eq:prod-ore}. One readily verifies that this makes $R[X;\sigma,\delta]$ a non-associative ring.  

In \cite{NOR18}, the authors then showed that $R[X;\sigma,\delta]$ is indeed a non-associative Ore extension of $R$ with $x=X$ \cite[Proposition 3.2]{NOR18}, and that every non-associative Ore extension of $R$ is isomorphic to a generalized  polynomial ring over $R$ \cite[Proposition 3.3]{NOR18}. If $\delta=0$, then $R[X;\sigma,\delta]$ is a non-associative skew polynomial ring, and the multiplication is given by \eqref{eq:skew-multiplication}. If $\sigma=\id_R$ and $\delta=0$, then $R[X;\sigma,\delta]$ is the ordinary polynomial ring $R[X]$.

\section{Flipped non-associative Ore extensions}\label{sec:flipped-non-associative-ore-extensions}
We begin with a general discussion, introducing the notion of \emph{flipped non-as\-so\-cia\-tive polynomial rings}. Let $R$ be a non-associative ring. The additive group $R[X]$ of left polynomials $\sum_{i\in\mathbb{N}}r_iX^i$ with coefficients $r_i\in R$ can be equipped with a natural left $R$-module structure in the following fashion: for any $r\in R$, 
\[
    r\Big(\sum_{i\in\mathbb{N}} r_iX^i\Big)=\sum_{i\in\mathbb{N}} (rr_i)X^i.
\]
By a \emph{non-associative polynomial ring (in one indeterminate) over} $R$, we shall mean a multiplication on the left $R$-module $R[X]$ that turns $R[X]$ into a (not necessarily unital) non-associative ring and which is compatible with the $R$-module structure on $R[X]$ in the following sense:
\begin{equation}
(rX^m)(sX^n)=\sum_{i\in\mathbb{N}}(r\ell_{m,n;i}(s))X^i, \label{eq:prod-general}
\end{equation}
for any $r,s\in R$ and some additive maps ${\ell_{m,n;i}: R\to R}$ ($m,n,i\in\mathbb{N}$). It is easy to see that there is a 1-1 correspondence between the set of non-associative polynomial rings over $R$ and the set of all collections of such additive maps subject to the condition that for every $m,n\in \mathbb{N}$ and $s\in R$, there exist only finitely many $i\in\mathbb{N}$ such that $\ell_{m,n;i}(s)\neq 0$.

Let $\tau\colon R\times R\to R\times R$ be the \emph{flip map}, i.e.,  $\tau$ is defined by $\tau(r,s)=(s,r)$ for arbitrary $r,s\in R$. We may now use $\tau$ to ``flip'' the above multiplication depending on whether $n$ is even or odd to get a new non-associative ring structure on the same underlying additive group $R[X]$:

\begin{equation}
(rX^m)(sX^n)=\sum_{i\in\mathbb{N}}\tau_n(r,\ell_{m,n;i}(s))X^i. \label{eq:prod-flipped}
\end{equation}
Here, $\tau_n$ is defined as the composition of the multiplication on $R$ and $\tau^n$, so that 
\[
\tau_{n}(r,s)=\begin{cases}rs&\text{if } n \text{ is even,}\\
sr&\text{if } n \text{ is odd}.
\end{cases}
\]
We call the resulting rings \emph{flipped non-associative polynomial rings}:

\begin{definition}[Flipped non-associative polynomial ring]\label{def:flipped-ring} Let $S$ be a non-associative polynomial ring over a non-associative ring $R$, defined by a collection of additive maps $\ell_{m,n;i}$ ($m,n,k\in\mathbb{N}$) satisfying \eqref{eq:prod-general}. The \emph{flipped non-associative polynomial ring of $S$}, written $S^{\fl}$, is the additive group $S$ equipped with the multiplication defined by \eqref{eq:prod-flipped}.
\end{definition}

 We now apply the notion of flipped non-associative polynomial rings to the class of generalized polynomial rings introduced in \autoref{subsec:non-assoc-ore}.
 
\begin{definition}[Flipped generalized polynomial ring]\label{def:flipped-poly} The flipped non-associative polynomial ring $R[X;\sigma,\delta]^{\fl}$ of a generalized polynomial ring $R[X;\sigma,\delta]$ over a non-associative ring $R$ is called a \emph{flipped generalized polynomial ring over $R$}.
\end{definition}
 
Note that the multiplication in the flipped generalized polynomial ring $R[X;\sigma,\delta]^{\fl}$ satisfies the identity 
\begin{equation}
(rX^m)(sX^n)=\sum_{i\in\mathbb{N}}\tau_n(r,\pi_i^m(s))X^{i+n}, \label{eq:prod-flip}
\end{equation}
for arbitrary $r,s\in R$. Here, the $\pi$-functions $\pi_i^m$ on $R$ are defined just as for generalized polynomial rings. By identifying any $r\in R$ with $rX^0\in R[X;\sigma,\delta]^{\fl}$, we see that $R[X;\sigma,\delta]^{\fl}$ naturally contains $R$ as a subring, and that $R[X;\sigma,\delta]^{\fl}$ is a free left $R$-module.

The next example is perhaps the simplest example of a flipped non-associative polynomial ring.

\begin{example}
If $R$ is a non-associative ring, then the multiplication in the flipped generalized polynomial ring $R[X]^{\fl}=R[X;\id_R,0]^{\fl}$ satisfies the identity
\[
(rX^m)(sX^n)=\tau_n(r,s)X^{m+n}
\]
for any $m,n\in\mathbb{N}$ and $r,s\in R$. We note that $R[X]^{\fl}$ is commutative if and only if $R$ is commutative, and that $R[X]^{\fl}$ is associative if and only if $R$ is associative and commutative; in both cases, $R[X]^{\fl}=R[X]$.
\end{example}

Continuing, we note that we may split the $\pi$-functions of any (flipped) generalized polynomial ring into two disjoint parts: those that contain $\sigma$ innermost (outermost) and those that contain $\delta$ innermost (outermost). If we denote by $\delta_i^m$ the Kronecker delta,
\[
 \delta_i^m\colonequals\begin{cases}0&\text{if } i\neq m,\\ 1&\text{if } i=m,\end{cases}	
\]
then the following lemma is immediate:
 
\begin{lemma}\label{lem:pi-sigma-delta-commute}
 If $R$ is a non-associative ring with additive maps $\sigma$ and $\delta$ where $\sigma(1)=1$ and $\delta(1)=0$, then the following equalities hold for any $i,m,n\in\mathbb{N}$ and $r,s\in R$:
\begin{enumerate}[label=\upshape(\roman*)]
	\item $\pi_i^m(1)=\delta_i^m$;\label{it:kronecker}
	\item $\pi_i^{m+1}(s)=\pi_{i-1}^m(\sigma(s))+\pi_i^m(\delta(s)) = \sigma(\pi_{i-1}^m(s))+\delta(\pi_i^m(s))$;\label{it:pi-functions-split}	
	\item $\sigma(\tau_n(r,s))=\tau_n(\sigma(r),\sigma(s))$ if and only if $\sigma$ is an endomorphism;\label{it:sigma-tau-commute-endomorphism}
	\item $\sigma(\tau_n(r,s))=\tau_n(\sigma(s),\sigma(r))$ if and only if $\sigma$ is an anti-endomorphism.\label{it:sigma-tau-commute-anti-endomorphism}
\end{enumerate}	
\end{lemma}
Next we turn to the notion of \emph{flipped non-associative Ore extensions}. As in the case of non-associative Ore extensions (see \autoref{def:non-assoc-ore}), we wish to define the flipped ditto in such a way that they correspond to the above flipped generalized polynomial rings. We suggest the following definition:  
\begin{definition}[Flipped non-associative Ore extension]\label{def:flip-ore}Let $S$ be a non-associative ring, and let $R$ be a subring of $S$ containing the multiplicative identity element $1$ of $S$. Then $S$ is called a \emph{flipped non-associative Ore extension of $R$} if there is an element $x\in S$ such that the following axioms hold:
 \begin{enumerate}[label=\upshape(F\arabic*)]
  \item The following identities hold for any $m,n\in\mathbb{N}$ and $r,s\in R$:\\$(rx^{m+1})(sx^n)=((rx^m)(\sigma(s)x^n))x+(rx^m)(\delta(s)x^n)$, \quad
 $r(sx^n)=\tau_n(r,s)x^n$, where $\sigma$ and $\delta$ are additive maps such that $xr=\sigma(r)x+\delta(r)$ for all $r\in R$;\label{it:flip3}
 \item $S$ is a free left $R$-module with basis $\{1,x,x^2,\ldots\}$.\label{it:flip1}
 \end{enumerate}
 \end{definition}

\begin{remark}
Axiom \ref{it:flip3} implies that $xR\subseteq Rx + R$, and that the element $x$ is power associative, so that $x^n$ is well defined for any $n\in\mathbb{N}$. 
\end{remark}
  
 \begin{proposition}\label{prop:generalized-is-flip}
 Let $R$ be a non-associative ring with additive maps $\sigma$ and $\delta$ where $\sigma(1)=1$ and $\delta(1)=0$. Then $R[X;\sigma,\delta]^{\fl}$ is a flipped non-associative Ore extension of $R$ with $x=X$.
 \end{proposition}

\begin{proof}
That \ref{it:flip1} holds in $R[X;\sigma,\delta]^{\fl}$ is immediate from the previous discussion. From \eqref{eq:prod-flip}, $Xr=(1X)(rX^0)=\sum_{i\in\mathbb{N}}\tau_0(1,\pi_i^1(r))X^{i+0}=\sigma(r)X+\delta(r)$ for any $r\in R$. By using \ref{it:kronecker} in \autoref{lem:pi-sigma-delta-commute}, for any $m\in\mathbb{N}$ and $r\in R$,
 \begin{equation}
 	(rX^m)X\stackrel{\eqref{eq:prod-flip}}{=}\sum_{i\in\mathbb{N}}\tau_1(r,\pi_i^m(1))X^{i+1}=\sum_{i\in\mathbb{N}}(\pi_i^m(1)r)X^{i+1}\stackrel{\text{\ref{it:kronecker}}}{=}\sum_{i\in\mathbb{N}}(\delta_i^mr)X^{i+1}=rX^{m+1}.\label{eq:(rX^m)X-identity}
 \end{equation}
Now, with the help of \ref{it:pi-functions-split} in \autoref{lem:pi-sigma-delta-commute}, we can prove that \ref{it:flip3} holds:
\begin{align*}
(rX^{m+1})(sX^n)&\stackrel{\hspace{1pt}\eqref{eq:prod-flip}\hspace{1pt}}{=}\sum_{i\in\mathbb{N}}\tau_n(r,\pi_i^{m+1}(s))X^{i+n}\\
&\stackrel{\text{\ref{it:pi-functions-split}}}{=}\sum_{i\in\mathbb{N}}\tau_n(r,\pi_{i-1}^m(\sigma(s)))X^{i+n}+\sum_{i\in\mathbb{N}}\tau_n(r,\pi_i^m(\delta(s)))X^{i+n}\\
&\stackrel{\hspace{1pt}\phantom{\eqref{eq:prod-flip}}\hspace{1pt}}{=}\sum_{i\in\mathbb{N}}\tau_n(r,\pi_i^m(\sigma(s))X^{i+n+1}+\sum_{i\in\mathbb{N}}\tau_n(r,\pi_i^m(\delta(s)))X^{i+n}\\
&\stackrel{\hspace{1pt}\eqref{eq:(rX^m)X-identity}\hspace{1pt}}{=}\Big(\sum_{j\in\mathbb{N}}\tau_n(r,\pi_i^m(\sigma(s))X^{i+n}\Big)X+\sum_{i\in\mathbb{N}}\tau_n(r,\pi_i^m(\delta(s)))X^{i+n}\\
&\stackrel{\hspace{1pt}\eqref{eq:prod-flip}\hspace{1pt}}{=}((rX^m)(\sigma(s)X^n))X+(rX^m)(\delta(s)X^n),\\
r(sX^n)&\stackrel{\hspace{1pt}\eqref{eq:prod-flip}\hspace{1pt}}{=}\sum_{i\in\mathbb{N}}\tau_n(r,\pi_i^0(s))X^{i+n}=\tau_n(r,\pi_0^0(s))X^n=\tau_n(r,s)X^n.\qedhere
\end{align*}
\end{proof}

 \begin{proposition}\label{prop:flip-is-isomorphic}
 Every flipped non-associative Ore extension of $R$ is isomorphic to a flipped generalized polynomial ring $R[X;\sigma,\delta]^{\fl}$.
 \end{proposition}
 
\begin{proof}
 The proof is similar to the proof of  \cite[Proposition 3.3]{NOR18}. 
 
 First, let $R$ be a non-associative ring and let $S$ be a flipped non-associative Ore extension of $R$ defined by $x$. By letting $n=0$ and $s=1$ in the first identity in \ref{it:flip3}, we have that for any $m\in\mathbb{N}$ and $r\in R$,
 \begin{equation}
    rx^{m+1}=(rx^m)x.\label{eq:x-right-associates} 
 \end{equation}
 We claim that for any $m,n\in\mathbb{N}$ and $r,s\in R$,
 \begin{equation}
 	(rx^m)(sx^n)=\sum_{i\in\mathbb{N}}\tau_n(r,\pi_i^m(s))x^{i+n}.\label{eq:prod-flip-induction}
 \end{equation}
We prove this by induction on $m$ where \eqref{eq:prod-flip-induction} is the induction hypothesis.

Base case ($m=0$): 
\[
	(rx^0)(sx^n)=r(sx^n)\stackrel{\text{\ref{it:flip3}}}{=}\tau_n(r,s)x^n=\sum_{i\in\mathbb{N}}\tau_n(r,\pi_i^0(s))x^{i+n}.
\]

Induction step $(m+1)$: 
\begin{align*}
	(rx^{m+1})(sx^n)&\stackrel{\text{\ref{it:flip3}}}{=}((rx^m)(\sigma(s)x^n))x+(rx^m)(\delta(s)x^{n})\\
&\stackrel{\eqref{eq:prod-flip-induction}}{\mathmakebox[\widthof{$\stackrel{\text{\ref{it:flip3}}}{=}$}]{=}}\Big(\sum_{i\in\mathbb{N}}\tau_n(r,\pi_i^m(\sigma(s)))x^{i+n}\Big)x+\sum_{i\in\mathbb{N}}\tau_n(r,\pi_i^m(\delta(s)))x^{i+n}\\
&\stackrel{\eqref{eq:x-right-associates}}{\mathmakebox[\widthof{$\stackrel{\text{\ref{it:flip3}}}{=}$}]{=}}\sum_{i\in\mathbb{N}}\tau_n(r,\pi_i^m(\sigma(s)))x^{i+n+1}+\sum_{i\in\mathbb{N}}\tau_n(r,\pi_i^m(\delta(s)))x^{i+n}\\
&\stackrel{\phantom{\text{\ref{it:flip3}}}}{=}\sum_{i\in\mathbb{N}}\tau_n(r,\pi_{i-1}^m(\sigma(s)))x^{i+n}+\sum_{i\in\mathbb{N}}\tau_n(r,\pi_i^m(\delta(s)))x^{i+n}\\
&\stackrel{\text{\ref{it:pi-functions-split}}}{\mathmakebox[\widthof{$\stackrel{\text{\ref{it:flip3}}}{=}$}]{=}}\sum_{i\in\mathbb{N}}\tau_n(r,\pi_i^{m+1}(s))x^{i+n},
\end{align*}
where \ref{it:pi-functions-split} refers to that in \autoref{lem:pi-sigma-delta-commute}.

Now, define a function $\phi\colon S\to R[X;\sigma,\delta]^{\fl}$ by the additive extension of the relations $\phi(rx^n)=rX^n$ for any $n\in\mathbb{N}$ and $r\in R$. Then $\phi$ is an isomorphism of additive groups, and moreover, for any $m,n\in\mathbb{N}$ and $r,s\in R$, 
\begin{align*}
\phi((rx^m) (sx^n)) &\stackrel{\eqref{eq:prod-flip-induction}}{=}
\phi\Big(\sum_{i\in\mathbb{N}}\tau_n(r,\pi_i^m(s))x^{i+n}\Big) = \sum_{i\in\mathbb{N}}\tau_n(r,\pi_i^m(s))X^{i+n} \\
&\stackrel{\eqref{eq:prod-flip}}{=} (rX^m) (sX^n)=\phi(rx^m)\phi(sx^n),
\end{align*}
so $\phi$ is an isomorphism of rings.
\end{proof}
 
\begin{remark}\label{remark: graded ring}
If $\sigma$ and $\delta$ satisfy the relation $\sigma\circ\delta + \delta\circ\sigma=0$, then the flipped generalized polynomial ring $S\colonequals R[X;\sigma,\delta]^{\fl}$ of $R$ can naturally be considered as a $\mathbb{Z}_2$-graded ring. More precisely, we have $S=S_0\oplus S_0X$, where $S_0$ is the subring of $S$ generated by $R$ and $X^2$. Furthermore, the ring $S_0$ is isomorphic to the generalized polynomial ring $R[Y;\sigma^2,\delta^2]$, where $Y$ is the image of $X^2$.  
\end{remark}

Next, we give some properties of flipped generalized polynomial rings.

\begin{proposition}\label{prop:X-nuclei}
Let $R$ be a non-associative ring with additive maps $\sigma$ and $\delta$ where $\sigma(1)=1$ and $\delta(1)=0$. If $S=R[X;\sigma,\delta]^{\fl}$, then the following assertions hold:
\begin{enumerate}[label=\upshape(\roman*)]
	\item $X\in N_l(S)$ if and only if $\sigma$ is an endomorphism and $\delta$ is both a left and a right $\sigma$-derivation;\label{it:X-left-nucleus}
	\item $X\in N_m(S)$ if and only if $\im(\delta^n\circ\sigma)\subseteq C(R)$ for any $n\in\mathbb{N}$;\label{it:X-middle-nucleus}
	\item $X\in N_r(S)$ if and only if $R$ is commutative.\label{it:X-right-nucleus}
\end{enumerate}	
\end{proposition}

\begin{proof}
\ref{it:X-left-nucleus}: We first show that the conditions are necessary. If $r,s\in R$ are arbitrary, then in $S$,
\begin{align*}
X(rs)&=\sigma(rs)X+\delta(rs),\\
(Xr)s&=(\sigma(r)X+\delta(r))s=(\sigma(r)\sigma(s))X+\sigma(r)\delta(s)+\delta(r)s.
\end{align*}
Hence, by comparing coefficients, if $X\in N_l(S)$, then $\sigma$ is an endomorphism and $\delta$ is a left $\sigma$-derivation. Similarly,
\begin{align*}
X(s(rX))&=X((rs)X)=\sigma(rs)X^2+\delta(rs)X,\\
(Xs)(rX)&=(\sigma(s)X+\delta(s))(rX)=(\sigma(r)\sigma(s))X^2+(\delta(r)\sigma(s)+r\delta(s))X,	
\end{align*}
so if $X\in N_l(S)$, then $\delta$ must also be a right $\sigma$-derivation. Next, we show that the conditions are sufficient. To this end, assume that $\sigma$ is an endomorphism and that $\delta$ is both a left and a right $\sigma$-derivation. Then, for any $n\in\mathbb{N}$ and $r,s\in R$,
\begin{align}
\delta(\tau_n(r,s))&=\begin{cases}\delta(rs)&\text{if } n \text{ is even},\\
\delta(sr)&\text{if } n \text{ is odd}.
\end{cases}	
=\begin{cases}\sigma(r)\delta(s)+\delta(r)s&\text{if } n \text{ is even},\\
\delta(s)\sigma(r)+s\delta(r)&\text{if } n \text{ is odd}.
\end{cases}\label{eq:delta-tau-commute}\\
&=\tau_n(\sigma(r),\delta(s))+\tau_n(\delta(r),s).\nonumber
\end{align}
It is sufficient to show that $X((rX^m)(sX^n))=(X(rX^m))(sX^n)$ for any $m,n\in\mathbb{N}$ and $r,s\in R$. By using \eqref{eq:delta-tau-commute} together with \ref{it:pi-functions-split} and \ref{it:sigma-tau-commute-endomorphism} in \autoref{lem:pi-sigma-delta-commute},

\allowdisplaybreaks
\begin{align*}
X((rX^m)(sX^n))\stackrel{\phantom{\text{\ref{it:sigma-tau-commute-endomorphism}}}}{=}&X\sum_{i\in\mathbb{N}}\tau_n(r,\pi_i^m(s))X^{i+n}=\sum_{i\in\mathbb{N}}X(\tau_n(r,\pi_i^m(s))X^{i+n})\\
\stackrel{\phantom{\text{\ref{it:sigma-tau-commute-endomorphism}}}}{=}&\sum_{i\in\mathbb{N}}\sigma(\tau_n(r,\pi_i^m(s)))X^{i+n+1}+\sum_{i\in\mathbb{N}}\delta(\tau_n(r,\pi_i^m(s)))X^{i+n},\\ 
(X(rX^m))(sX^n)\stackrel{\phantom{\text{\ref{it:sigma-tau-commute-endomorphism}}}}{=}&(\sigma(r)X^{m+1}+\delta(r)X^m)(sX^n)\\
\stackrel{\phantom{\text{\ref{it:sigma-tau-commute-endomorphism}}}}{=}&\sum_{i\in\mathbb{N}}\tau_n(\sigma(r),\pi_i^{m+1}(s))X^{i+n}+\sum_{i\in\mathbb{N}}\tau_n(\delta(r),\pi_i^{m}(s))X^{i+n}\\
\stackrel{\text{\ref{it:pi-functions-split}}}{\mathmakebox[\widthof{$\stackrel{\text{\ref{it:sigma-tau-commute-endomorphism}}}{=}$}]{=}}&\sum_{i\in\mathbb{N}}\tau_n(\sigma(r),\sigma(\pi_{i-1}^{m}(s)))X^{i+n}+\sum_{i\in\mathbb{N}}\tau_n(\sigma(r),\delta(\pi_i^{m}(s)))X^{i+n}\\
&+\sum_{i\in\mathbb{N}}\tau_n(\delta(r),\pi_i^{m}(s))X^{i+n}\\
\stackrel{\phantom{\text{\ref{it:sigma-tau-commute-endomorphism}}}}{=}&\sum_{i\in\mathbb{N}}\tau_n(\sigma(r),\sigma(\pi_{i}^{m}(s)))X^{i+n+1}+\sum_{i\in\mathbb{N}}\tau_n(\sigma(r),\delta(\pi_i^{m}(s)))X^{i+n}\\
&+\sum_{i\in\mathbb{N}}\tau_n(\delta(r),\pi_i^{m}(s))X^{i+n}\\
\stackrel{\text{\ref{it:sigma-tau-commute-endomorphism}}}{=}&
\sum_{i\in\mathbb{N}}\sigma(\tau_n(r,\pi_{i}^{m}(s)))X^{i+n+1}+\sum_{i\in\mathbb{N}}\tau_n(\sigma(r),\delta(\pi_i^{m}(s)))X^{i+n}\\
&+\sum_{i\in\mathbb{N}}\tau_n(\delta(r),\pi_i^{m}(s))X^{i+n}\\
\stackrel{\eqref{eq:delta-tau-commute}}{\mathmakebox[\widthof{$\stackrel{\text{\ref{it:sigma-tau-commute-endomorphism}}}{=}$}]{=}}&
\sum_{i\in\mathbb{N}}\sigma(\tau_n(r,\pi_i^m(s)))X^{i+n+1}+\sum_{i\in\mathbb{N}}\delta(\tau_n(r,\pi_i^m(s)))X^{i+n}.
\end{align*}

\noindent\ref{it:X-middle-nucleus}: By using \ref{it:pi-functions-split} in \autoref{lem:pi-sigma-delta-commute}, for any $m,n\in\mathbb{N}$ and $r,s\in R$,
\begin{align*}
((rX^m)X)(sX^n)&\stackrel{\phantom{\text{\ref{it:pi-functions-split}}}}{=}(rX^{m+1})(sX^n)=\sum_{i\in\mathbb{N}}\tau_n(r,\pi_i^{m+1}(s))X^{i+n}\\
&\stackrel{\text{\ref{it:pi-functions-split}}}{=}\sum_{i\in\mathbb{N}}\tau_n(r,\pi_{i-1}^{m}(\sigma(s)))X^{i+n}+\sum_{i\in\mathbb{N}}\tau_n(r,\pi_i^{m}(\delta(s)))X^{i+n}\\
&\stackrel{\phantom{\text{\ref{it:pi-functions-split}}}}{=}\sum_{i\in\mathbb{N}}\tau_n(r,\pi_{i}^{m}(\sigma(s)))X^{i+n+1}+\sum_{i\in\mathbb{N}}\tau_n(r,\pi_i^{m}(\delta(s)))X^{i+n}\\
&\stackrel{\phantom{\text{\ref{it:pi-functions-split}}}}{=}\sum_{i\in\mathbb{N}}\tau_{n+1}(\pi_{i}^{m}(\sigma(s)),r)X^{i+n+1}+\sum_{i\in\mathbb{N}}\tau_n(r,\pi_i^{m}(\delta(s)))X^{i+n},\\
(rX^m)(X(sX^n))&\stackrel{\phantom{\text{\ref{it:pi-functions-split}}}}{=}(rX^m)(\sigma(s)X^{n+1}+\delta(s)X^n)\\
&\stackrel{\phantom{\text{\ref{it:pi-functions-split}}}}{=}\sum_{i\in\mathbb{N}}\tau_{n+1}(r,\pi_i^m(\sigma(s)))X^{i+n+1}+\sum_{i\in\mathbb{N}}\tau_n(r,\pi_i^m(\delta(s)))X^{i+n}.
\end{align*}

Hence $X\in N_m(S)$ if and only if 
\[
\sum_{i\in\mathbb{N}}\tau_{n+1}(\pi_i^m(\sigma(s)),r)X^{i+n+1}=\sum_{i\in\mathbb{N}}\tau_{n+1}(r,\pi_i^m(\sigma(s)))X^{i+n+1}.
\]
By comparing coefficients, the above equality is equivalent to $\tau_{n+1}(\pi_i^m(\sigma(s)),r)=\tau_{n+1}(r,\pi_i^m(\sigma(s)))$, which in turn is equivalent to $r\pi_i^m(\sigma(s))=\pi_i^m(\sigma(s))r$ for any $i,m\in\mathbb{N}$ and $r,s\in R$. We claim that the latter equality is equivalent to $\im(\delta^n\circ\sigma)\subseteq C(R)$ for any $n\in\mathbb{N}$. By setting $m=n$ and $i=0$ in $r\pi_i^m(\sigma(s))=\pi_i^m(\sigma(s))r$, we see that the conditions are necessary. To show sufficiency, we note that each term in $\pi_i^m(\sigma(s))$ must contain $\delta^n\circ\sigma$ outermost for some $n \leq m-i$, and so we are done.\\

\noindent\ref{it:X-right-nucleus}: If $r,s\in R$ are arbitrary, then $r(sX)=(sr)X$, so $r(sX)=(rs)X$ if and only if $rs=sr$. Hence, if $X\in N_r(S)$, then $R$ is commutative. To prove the converse, assume that $R$ is commutative. Then $\tau_n(r,s)=rs$ for any $n\in\mathbb{N}$ and $r,s\in R$, so $S$ is a generalized polynomial ring $R[X;\sigma,\delta]$. Hence $S$ is a non-associative Ore extension of $R$ with $x=X$, so by \ref{it:no3} in \autoref{def:non-assoc-ore}, $X\in N_r(S)$.
\end{proof}

\begin{corollary}\label{cor:flip-is-non-associative-ore}
Let $R$ be a non-associative ring with additive maps $\sigma$ and $\delta$ where $\sigma(1)=1$ and $\delta(1)=0$. Then $R[X;\sigma,\delta]^{\fl}=R[X;\sigma,\delta]$ if and only if $R$ is commutative.	
\end{corollary}

\begin{proof}
Let $S=R[X;\sigma,\delta]$. If $R$ is commutative, then $\tau_n(r,s)=rs$ for any $n\in\mathbb{N}$ and $r,s\in R$, and so $S^{\fl}=S$. If $S^{\fl}=S$, then $X\in N_r(S)=N_r(S^{\fl})$ which by \ref{it:X-right-nucleus} in \autoref{prop:X-nuclei} is equivalent to $R$ being commutative. 
\end{proof}

\begin{proposition}\label{prop:associativity-equivalence}
Let $S=R[X;\sigma,\delta]^{\fl}$ be a flipped  generalized polynomial ring over a non-associative ring $R$. Then the following conditions are equivalent: 
\begin{enumerate}[label=\upshape(\roman*)]
	\item $R$ is associative and commutative, $\sigma$ is an endomorphism and $\delta$ is a left $\sigma$-derivation;\label{it:assoc-cond1}
	\item $S$ is an ordinary generalized polynomial ring;\label{it:assoc-cond2}
	\item $S$ is associative.\label{it:assoc-cond3}
\end{enumerate}	
\end{proposition}

\begin{proof}
\ref{it:assoc-cond1}$\implies$\ref{it:assoc-cond2}: If $R$ is commutative, then $R[X;\sigma,\delta]^{\fl}=R[X;\sigma,\delta]$ by \autoref{cor:flip-is-non-associative-ore}, and so the result follows immediately.\\

\noindent\ref{it:assoc-cond2}$\implies$\ref{it:assoc-cond3}: For a proof of the associativity of the ordinary generalized polynomial ring $R[X;\sigma,\delta]$, see \cite{Nys13} or \cite{Ric14}.\\

\noindent\ref{it:assoc-cond3}$\implies$\ref{it:assoc-cond1}: Assume that $S$ is associative. Then so is $R$, since $R$ is a subring of $S$. Also, $X\in N_l(S)\cap N_r(S)$, so the result follows from \ref{it:X-left-nucleus} and \ref{it:X-right-nucleus} in \autoref{prop:X-nuclei}.
\end{proof}

\section{The Cayley--Dickson construction}\label{sec:cayley-dickson}
Let $A$ be a non-associative algebra over an associative and commutative ring $K$ of scalars. We say that a $K$-linear map $*\colon A\to A$, also written as $a\mapsto a^*$, is an \emph{involution} of $A$ if $(ab)^*=b^*a^*$ and $(a^*)^*=a$ hold for any $a,b\in A$. In particular, an involution is an anti-automorphism. We note that $1^*=1$ since for any $a\in A$, $1^*a^*=(a1)^*=a^*=(1a)^*=a^*1^*$, so by the surjectivity of $*$ and the uniqueness of $1$, $1^*=1$. 

We refer to a non-associative algebra with an involution $*$ as a non-associative \emph{$*$-algebra}. If $A$ and $B$ are non-associative $*$-algebras over the same ring $K$, then a \emph{$*$-homomorphism} $f\colon A\to B$ is a $K$-algebra homomorphism that is compatible with the involutions on $A$ and $B$, i.e., with some abuse of notation,  $f(a^*)=f(a)^*$ for any $a\in A$. A bijective $*$-homomorphism is called a $*$-isomorphism, and the corresponding algebras are said to be \emph{isomorphic as $*$-algebras}. For an introduction to $*$ -algebras, we refer the reader to \cite[Section 2.2]{McC04}.

From the above discussion, we note that any non-associative $*$-algebra $A$ naturally gives rise to a flipped non-associative skew polynomial ring $A[X;*]^{\fl}\colonequals A[X;*,0]^{\fl}$. From \eqref{eq:prod-flip}, the multiplication in $A[X;*]^{\fl}$ is thus given by
\begin{equation}
(aX^m)(bX^n)=\tau_n(a,*^m(b))X^{m+n}\label{eq:star-skew-product}
\end{equation}
for arbitrary $m,n\in\mathbb{N}$ and $a,b\in A$. The next proposition shows that $A[X;*]^{\fl}$ can also be made into a $*$-algebra.

\begin{proposition}\label{prop: extendion of star}
Let $A$ be a non-associative $*$-algebra over an associative and commutative ring $K$. Then the maps $\alpha,\beta:A[X;*]^{\fl}\to A[X;*]^{\fl}$ defined by
\begin{align*}
\alpha(a_0+a_1X+a_2X^2+\cdots)&=\sum_{i\in\mathbb{N}}(-1)^i*^{i+1}(a_i)X^i,\\
\beta(a_0+a_1X+a_2X^2+\cdots)&=\sum_{i\in\mathbb{N}}*^{i+1}(a_i)X^i
\end{align*}
are involutions of $A[X;*]^{\fl}$, which extend the involution $*$ of $A$. Furthermore, if $*$ is nontrivial, $A$ does not have zero divisors, and $2A\neq 0$, then $\alpha$ and $\beta$ are the only involutions of $A[X;*]^{\fl}$ that extend the involution $*$ of $A$.
\end{proposition}

\begin{proof}
First, we show that $\alpha$ is an  involution of $ A[X;*]^{\fl}$.  It is clear from the definition that $\alpha$ is $K$-linear and satisfies $\alpha^2(p)=p$ for all $p\in A[X;*]^{\fl}$. Hence we only need to show that $\alpha(pq)=\alpha(q)\alpha(p)$ for any $p,q\in A[X;*]^{\fl}$. We claim that for any $m,n\in\mathbb{N}$ and $a,b\in A$,  
\begin{align}
\tau_m(a,b)&=\tau_{m+1}(b,a),\label{eq:tau-commutation}\\
	*^{m+n}(\tau_m(*^n(a),b))&=\tau_n(*^m(a),*^{m+n}(b)).\label{eq:sigma-tau-commutation}
\end{align}
The first equality is immediate. Regarding the last equality, we note that if $m+n$ is even, then $*^{m+n}$ is the identity map, and either $m$ and $n$ are both even, or $m$ and $n$ are both odd, so that $\tau_m=\tau_n$. If $m+n$ is odd, then $*^{m+n}$ is an anti-endomorphism and either $m$ is even and $n$ is odd, or $m$ is odd and $n$ is even, so that $\tau_m(a,b)=\tau_n(b,a)$. The result now follows from \ref{it:sigma-tau-commute-anti-endomorphism} in \autoref{lem:pi-sigma-delta-commute}. Now, write $p=\sum_{i\in\mathbb{N}}p_iX^i$ and $q=\sum_{j\in\mathbb{N}}q_jX^j$ where $p_i,q_j\in A$. By definition, $\alpha(p)=\sum_{i\in\mathbb{N}}(-1)^i*^{i+1}(p_i)X^i$, so 
\begin{align*}
\alpha(pq)&\stackrel{\phantom{\eqref{eq:tau-commutation}}}{=}\alpha\Big(\Big(\sum_{i\in\mathbb{N}}p_iX^i\Big)\Big(\sum_{j\in\mathbb{N}}q_jX^j\Big)\Big)\stackrel{\eqref{eq:star-skew-product}}{=}\alpha\Big(\sum_{i,j\in\mathbb{N}}\tau_j(p_i,*^i(q_j))X^{i+j}\Big) \\
&\stackrel{\phantom{\eqref{eq:tau-commutation}}}{=}\sum_{i,j\in\mathbb{N}}(-1)^{i+j}*^{i+j+1}(\tau_j(p_i,*^i(q_j)))X^{i+j}\\
&\stackrel{\eqref{eq:tau-commutation}}{=}\sum_{i,j\in\mathbb{N}}(-1)^{i+j}*^{i+j+1}(\tau_{j+1}(*^i(q_j),p_i))X^{i+j}\\
&\stackrel{\eqref{eq:sigma-tau-commutation}}{=}\sum_{i,j\in\mathbb{N}}(-1)^{i+j}\tau_{i}(*^{j+1}(q_j),*^{i+j+1}(p_i))X^{i+j}\\
&\stackrel{\eqref{eq:star-skew-product}}{=}\Big(\sum_{j\in\mathbb{N}}(-1)^j*^{j+1}(q_j)X^j\Big)\Big(\sum_{i\in\mathbb{N}}(-1)^i*^{i+1}(p_i)X^i\Big)=\alpha(q)\alpha(p).
\end{align*}
In a similar fashion, one can show that $\beta$ is an involution of $A[X;*]^{\fl}$. 

To prove the last statement, assume that $*$ is nontrivial, $A$ does not have zero divisors, and $2A\neq 0$. Let $\gamma$ be an involution of $ A[X;*]^{\fl}$ that extends $*$ of $A$. Our goal is to show that $\gamma$ coincides with either $\alpha$ or $\beta$.  Since $\gamma^2(X)=X$, an argument involving counting degrees reveals that $\gamma(X)$ must be of degree $1$. So, $\gamma(X)=a+bX$ for some $a,b\in A$. The identity $\gamma^2(X)=X$ gives
\[
X=\gamma^2(X)=\gamma(a+bX)=a^*+\gamma(X)b^*=a^*+(a+bX)b^*=a^*+ab^*+b^2X.
\]
It follows that $a^*+ab^*=0$ and $b^2=1$. 
Since $A$ does not have zero divisors and $(b+1)(b-1)=0$, $b$ is either $1$ or $-1$. Applying $\gamma$ to the identity $Xc=c^*X$, where $c\in A$, yields $c^*(a+bX)=(a+bX)c$, from which it follows that $c^*a=ac$ for all $a\in A$. 
We distinguish two cases. 

\textbf{Case 1.} $b=-1$. From $a^*+ab^*=0$, it follows that $a^*=a$. We also have $c^*a=ac$ for all $a\in A$. We apply $\gamma$ to both sides of the identity $(cX)X=cX^2$. On the one hand, we have
\begin{align*}
\gamma((cX)X)&=\gamma(X)\gamma(cX)=\gamma(X)\big(\gamma(X)\gamma(c) \big) =(a-X)\big( (a-X)c^* \big)\\
&=(a-X)( ac^*-Xc^*)=(a-X)(ac^*-cX)\\
&=a(ac^*)-X(ac^*)-a(cX)+X(cX)\\
&=a(ac^*)-(ac^*)^*X-(ca)X+c^*X^2\\
&=a(ac^*)-(ca^*)X-(ca)X+c^*X^2=a(ac^*)-(2ca)X+c^*X^2.
\end{align*}
On the other hand, we have
\begin{align*}
\gamma(cX^2)&=\gamma(X^2)\gamma(c)=\big( \gamma(X)\gamma(X) \big)\gamma(c) =\big( (a-X)(a-X) \big)c^*\\
&=(a^2-Xa-aX+X^2)c^*=(a^2-a^*X-aX+X^2)c^*\\
&=(a^2-2aX+X^2)c^*=a^2c^*-(2aX)c^*+X^2c^*\\
&=a^2c^*-(2ac)X+c^*X^2=a^2c^*-(2c^*a)X+c^*X^2.
\end{align*}
Comparing the coefficients of the flipped polynomials, we see that $2c^*a=2ca$, or equivalently, $2(c^*-c)a=0$, for all $a\in A$. Since $A$ does not have zero divisors, $2A\neq0$, and $*$ is nontrivial, we conclude that $a=0$, that is, $\gamma=\alpha$.

\textbf{Case 2.} $b=1$. It follows from $a^*+ab^*=0$ that $a^*+a=0$. We also have $c^*a=ac$ for all $a\in A$. Setting $c=a$, we obtain $a^*a=a^2$ implying $2a^2=0$ because $a^*=-a$.  Since $A$ does not have zero divisors and $2A\neq 0$, we conclude that $a=0$, that is, $\gamma=\beta$. This completes the proof of the proposition. 
\end{proof}

In what follows, we consider $A[X;*]^{\fl}$ as a $*$-algebra using the involution $\alpha$ introduced in \autoref{prop: extendion of star}. Also, we follow McCrimmon's  conventions in \cite{McC85}.

Let $A$ be a non-associative $*$-algebra over an associative and commutative ring $K$, and assume that $\mu\in K$ is a cancellable scalar, meaning that if $\mu a=0$ for some $a\in A$, then $a=0$. Further assume that $K$ acts \emph{faithfully} on $A$, meaning that if $kA=0$ for some $k\in K$, then $k=0$. We can now construct a new non-associative $*$-algebra over $K$, the \emph{Cayley double of $A$}, written $\Cay(A,\mu)$. $\Cay(A,\mu)$ is defined as $A\oplus A$ with involution and product defined by $(a,b)^*=(a^*,-b)$ and $(a,b)(c,d)=(ac + \mu d^*b, da+bc^*)$ for arbitrary $a,b,c,d\in A$, respectively.

\begin{example}\label{ex:classical-Cayley-Dickson}
If we start with $*=\id_K$ on $K=\mathbb{R}$ and choose $\mu=-1$ and then double, we obtain $\mathbb{C}, \mathbb{H}, \mathbb{O},\ldots$. If we instead choose $\mu=+1$ and then double, we obtain the so-called \emph{split} versions $\mathbb{C}', \mathbb{H}', \mathbb{O}', \ldots$ of the above algebras:
\begin{align*}
\Cay(\mathbb{R},-1)&\cong \mathbb{C},& \Cay(\mathbb{R},+1)\phantom{'}&\cong  \mathbb{C}',\\
\Cay(\mathbb{C},-1)&\cong\mathbb{H},	&\Cay(\mathbb{C}',+1)&\cong  \mathbb{H}',\\
\Cay(\mathbb{H},-1)&\cong\mathbb{O}, &\Cay(\mathbb{H}',+1)&\cong  \mathbb{O}',\\
&\phantom{\vdots\vdots} \vdots && \phantom{\vdots\vdots} \vdots 
\end{align*}
\end{example}

There is a connection between $A[X;*]^{\fl}$ and the concept of a Cayley double.  More precisely, the Cayley double of $A$ is isomorphic to a quotient ring of $A[X;*]^{\fl}$ (see \autoref{thm:cayley-dickson-isomorphism}), and furthermore, $A[X;*]^{\fl}$ is isomorphic to the Cayley double of an appropriate ring (see \autoref{thm:cayley-dickson-isomorphism1}). 

Before the next theorem, let us remark that if $I$ is an ideal of $B\colonequals A[X;*]^{\fl}$ such that $I^*\subseteq I$, then the extended involution $*$ of $B$ induces a well-defined involution of $B/I$ by $[p]^*\colonequals [p^*]$ for any $p\in B$ and its equivalence class $[p]\in B/I$. In this way, $B/I$ naturally becomes a $*$-algebra. In particular, if $I=\langle X^2-\mu\rangle$, then by iteratively using that $(pq)^*=q^*p^*$ for any $p,q\in B$ and $(X^2-\mu)^*=X^2-\mu\in I$, we see that $I^*\subseteq I$. 

\begin{theorem}\label{thm:cayley-dickson-isomorphism}
$\Cay(A,\mu)$ and $A[X;*]^{\fl}/\langle X^2-\mu\rangle$ are isomorphic as $*$-algebras.
\end{theorem}

\begin{proof}
Define $B\colonequals A[X;*]^{\fl}$ and $I\colonequals\langle X^2-\mu\rangle$, and let $n\in\mathbb{N}$ and $a\in A$ be arbitrary. Since $aX^n$ is the product of $a$ with $X^n$, we have $[aX^n]=[a][X^n]$. Hence, by using the power associativity of $X$, 
\begin{align*}
[aX^{2n}]&=[a][X^{2n}]=[a][(X^2)^n]=[a][X^2]^n=[a][\mu]^n=[a\mu^n],\\
[aX^{2n+1}]&=[a][X^{2n+1}]=[a][X(X^2)^n]=[a]([X][X^2]^n)=[a]([X][\mu^n])=[a\mu^nX]. 
\end{align*}
Since $B$ is a free left $A$-module with basis $\{1,X,X^2,\ldots\}$, the above computations imply that any element in $B/I$ may be written as $[a'+b'X]$ for some $a',b'\in A$. Now, for arbitrary $a,b,c,d\in A$, 
\begin{align}\label{eq:cayley-prod}
[a+bX][c+dX]&=[ac+a(dX)+(bX)c+(bX)(dX)]\\
&=[ac+\tau_1(a,d)X+\tau_0(b,c^*)X+\tau_1(b,d^*)X^2]\nonumber\\
&=[ab+(da)X+(bc^*)X+(d^*b)X^2]=[ab+\mu d^*b+(da+bc^*)X].\nonumber
\end{align}
Define a map $\phi\colon B/I\to \Cay(A,\mu)$ by $[a+bX]\mapsto (a,b)$ and extend it $K$-linearly. Then $\phi$ is clearly surjective. Moreover, by comparing degrees, $[a+bX]=0\iff a+bX\in I\iff a=b=0$. Hence $\ker \phi=0$, so $\phi$ is injective. We  then have 
\begin{align*}
\phi([a+bX][c+dX])&\stackrel{\eqref{eq:cayley-prod}}{=}\phi([ab+\mu d^*b+(da+bc^*)X])=(ab+\mu d^*b, da+bc^*)\\
&\stackrel{\phantom{\eqref{eq:cayley-prod}}}{=}(a,b)(c,d)=\phi([a+bX])\phi([c+dX]),
\end{align*}
so $\phi$ is an isomorphism of $K$-algebras. Last, we show that $\phi$ is compatible with the involutions of $B/I$ and $\Cay(A,\mu)$: $\phi([a+bX]^*)=\phi([(a+bX)^*])=\phi([a^*-bX])=(a^*,-b)=\phi([a+bX])^*$.
\end{proof}

\begin{example}\label{ex:classical-Cayley-Dickson-quotient}This example follows immediately from \autoref{ex:classical-Cayley-Dickson} and \autoref{thm:cayley-dickson-isomorphism}.
\begin{align*}
\mathbb{C}&\cong\mathbb{R}[X]/\langle X^2+1\rangle,& \mathbb{C}'&\cong \mathbb{R}[X]/\langle X^2-1\rangle,\\
\mathbb{H}&\cong\mathbb{C}[X;*]/\langle X^2+1\rangle,&\mathbb{H}'&\cong\mathbb{C}'[X;*]/\langle X^2-1\rangle,\\
\mathbb{O}&\cong\mathbb{H}[X;*]^{\fl}/\langle X^2+1\rangle, &\mathbb{O}'&\cong\mathbb{H}'[X;*]^{\fl}/\langle X^2-1\rangle ,\\
&\phantom{\vdots\vdots} \vdots && \phantom{\vdots\vdots} \vdots 
\end{align*}
\end{example}

Let $A$ be a non-associative $*$-algebra over an associative and commutative ring $K$. Consider the generalized polynomial ring $A[t]=A[t;\id_A,0]$, which is naturally an algebra over the associative and commutative ring $K[t]$. Extending $*$ to $A[t]$ using $t^*=t$, we consider $A[t]$ as a non-associative $*$-algebra over $K[t]$. Since the element $t\in K[t]$ is a cancellable scalar, we can form the Cayley double $\Cay(A[t],t)$. 

\begin{theorem}\label{thm:cayley-dickson-isomorphism1}
$\Cay(A[t],t)$ and  $A[X;*]^{\fl}$ are isomorphic as $*$-algebras.
\end{theorem}
\begin{proof}
    We define a map $\phi:\Cay(A[t],t)\to A[X;*]^{\fl}$ as follows: \[\phi(p(t),q(t))=p(X^2)+q(X^2)X,\quad\text{for any } p(t),q(t)\in A[t].\] In the light of \autoref{remark: graded ring}, it is straightforward to check that $\phi$ is an isomorphism of $K$-algebras. We also see that $\phi$ is compatible with the involutions of $\Cay(A[t],t)$ and $A[X;*]^{\fl}$: $\phi((p(t),q(t))^*)=\phi(p(t)^*,-q(t))=p(X^2)^*-q(X^2)X=\phi(p(t),q(t))^*$.
\end{proof}

Recall that $A$ is called \emph{flexible} if $(a,b,a)=0$, and \emph{alternative} if $(a,a,b)=(b,a,a)=0$ for all $a,b\in A$. The criteria for $A[X;*]^{\fl}$ to inherit algebraic properties like the above from $A$ are given by the following two theorems (compare with \cite[Theorem 6.8]{McC85}):

\begin{theorem}\label{thm:Cayley-Dickson-properties}
If $B=A[X;*]^{\fl}$, then the following assertions hold:
\begin{enumerate}[label=\upshape(\roman*)]
\item The involution $\alpha$ on $B$ is trivial if and only if the involution $*$ on $A$ is trivial and $2A=0$;\label{it:cayley-dickson-properties1}
\item $B$ is commutative if and only if $A$ is commutative with trivial involution;\label{it:cayley-dickson-properties2}
\item $B$ is associative if and only if $A$ is associative and commutative;\label{it:cayley-dickson-properties3}
\item $B$ is flexible if and only if $A$ is flexible, $aa^*$ commutes with $A$, and $(a,b,c)=(a,b^*,c^*)$ for all $a,b,c\in A$; \label{it:cayley-dickson-properties4}
\item $B$ is alternative if and only if $A$ is alternative, $aa^*$ commutes with $A$, and $2a+a^*$ lies in the nucleus of $A$ for all $a\in A$. \label{it:cayley-dickson-properties5}
\end{enumerate}	
\end{theorem}

\begin{proof}
\ref{it:cayley-dickson-properties1}: Assume that $B$ has trivial involution. Since the involution of $A$ is the involution of $B$ restricted to $A$, it is also trivial. Also, for any $a\in A$, $(aX)^*=X^*a^*=-Xa^*=-aX$, so $(aX)^*=aX\iff(2a)X=0\iff 2a=0$. Now, assume instead that $A$ has trivial involution, $2A=0$, and let $p\in B$. Then $p=a_0+a_1X+a_2X^2+\cdots$ for some $a_0,a_1,a_2,\ldots\in A$ and $p^*=a_0^*-a_1X+a_2^*X+\cdots=a_0-a_1X+a_2X^2+\cdots$. Since $a_{2i+1}=-a_{2i+1}$ for any $i\in\mathbb{N}$, $p=p^*$.

For a different proof, one can use \autoref{thm:cayley-dickson-isomorphism1} together with \cite[Theorem 6.8 (i)]{McC85}.
\\

\noindent\ref{it:cayley-dickson-properties2}: From \autoref{cor:flip-is-non-associative-ore}, it is immediate that $B$ is commutative if and only if $B=A[X;\id_A]^{\fl}$ and $A$ is commutative. Alternatively, one can use \autoref{thm:cayley-dickson-isomorphism1} together with \cite[Theorem 6.8 (ii)]{McC85}.\\

\noindent\ref{it:cayley-dickson-properties3}: This follows immediately from \ref{it:assoc-cond1} and \ref{it:assoc-cond3} in \autoref{prop:associativity-equivalence}. For an alternative proof, one can again use \autoref{thm:cayley-dickson-isomorphism1} together with \cite[Theorem 6.8 (iii)]{McC85}.\\

\noindent\ref{it:cayley-dickson-properties4}: By \autoref{thm:cayley-dickson-isomorphism1}, $B$ is flexible if and only if $\Cay(A[t],t)$ is flexible, which by \cite[Theorem 6.8 (iv)]{McC85} is equivalent to the conditions that $A[t]$ is flexible, $p(t)p(t)^*$ commutes with $A[t]$, and $(p(t),q(t),r(t))=(p(t),q(t)^*,r(t)^*)$ for all $p(t),q(t),r(t)\in A[t]$. Since $t\in$ $Z(A[t])$ and $t^*=t$, a moment's reflection shows that these conditions hold over $A[t]$ if and only if they hold over $A$.\\ 

\noindent\ref{it:cayley-dickson-properties5}: By \autoref{thm:cayley-dickson-isomorphism1}, $B$ is alternative if and only if $\Cay(A[t],t)$ is alternative, which by \cite[Theorem 6.8 (v)]{McC85} is equivalent to the conditions that $A[t]$ is alternative, $p(t)p(t)^*$ commutes with $A[t]$, and $2p(t)+p(t)^*\in N(A[t])$ for all $p(t)\in A[t]$. Since $t\in Z(A[t])$ and $t^*=t$, it is easy to see that these conditions hold over $A[t]$ if and only if they hold over $A$.
\end{proof}

We define $C_*(A)$ as $\{a\in C(A)\colon a^*=a\}$ and $Z_*(A)$ as $\{a\in Z(A)\colon a^*=a\}$. With these notations, we have the following theorem:

\begin{theorem}\label{thm:cayley-dickson-centers}
If $B=A[X;*]^{\fl}$, then the following equalities hold:
\begin{enumerate}[label=\upshape(\roman*)]
\item $\begin{aligned}[t]C(B)=\ &\big\{\textstyle\sum_{i\in\mathbb{N}}a_iX^i\colon a_i\in C_*(A), a_{2i+1}b^*=a_{2i+1}b\ \forall b\in A, i\in\mathbb{N}\big\};\end{aligned}$\label{it:Cayley-commuter}
\item $\begin{aligned}[t]N_l(B)=\ &N_r(B)=\big\{\textstyle\sum_{i\in\mathbb{N}}a_iX^i\colon a_{2i}\in Z(A), a_{2i+1}\in C(A)\cap N_m(A),\\
&(a_{2i+1}b)c = a_{2i+1}(cb), (bc)a_{2i+1}=c(ba_{2i+1})\ \forall b,c\in A, i\in\mathbb{N}\big\};\end{aligned}$\label{it:Cayley-left-right-nucleus}
\item $\begin{aligned}[t]N_m(B)=\ &\big\{\textstyle\sum_{i\in\mathbb{N}}a_iX^i\colon a_{2i}\in C(A)\cap N_m(A),a_{2i+1}\in C(A),\\
&(a_{2i+1}b)c =(a_{2i+1}c)b, b(ca_{2i+1})=c(ba_{2i+1})\ \forall b,c\in A, i\in\mathbb{N}\big\};\end{aligned}$\label{it:Cayley-middle-nucleus}
\item $\begin{aligned}[t]N(B)=\ &\big\{\textstyle\sum_{i\in\mathbb{N}}a_iX^i\colon a_{i}\in Z(A), a_{2i+1}[A,A]=0\ \forall i\in\mathbb{N}\big \};\end{aligned}$\label{it:Cayley-nucleus}
\item $\begin{aligned}[t]Z(B)=\ &\big\{\textstyle\sum_{i\in\mathbb{N}}a_iX^i\colon a_i\in Z_*(A), a_{2i+1}b^*=a_{2i+1}b\ \forall b\in A, i\in\mathbb{N}\big\}.
\end{aligned}$\label{it:Cayley-center}
\end{enumerate}
\end{theorem}

\begin{proof}
\ref{it:Cayley-commuter}: We note that $\sum_{i\in\mathbb{N}}a_iX^i\in C(B)$, where $a_i\in A$ is arbitrary, if and only if $a_iX^i\in C(B)$ for any $i\in\mathbb{N}$. Hence, to show that the conditions are necessary, assume that $a_iX^i\in C(B)$ and let $b\in A$ be arbitrary. Then $[a_{2i}X^{2i},b]=[a_{2i+1}X^{2i+1},b]=0$, and by \eqref{eq:star-skew-product}, $[a_{2i}X^{2i},b]=(a_{2i}b-ba_{2i})X^{2i}$ and $[a_{2i+1}X^{2i+1},b]=(a_{2i+1}b^*-a_{2i+1}b)X^{2i+1}$. By comparing coefficients, $a_{2i}\in C(A)$ and $a_{2i+1}b^*=a_{2i+1}b$ for any $i\in\mathbb{N}$. Similarly, $[a_{2i}X^{2i},bX]=[a_{2i+1}X^{2i+1},b^*X]=0$, and by \eqref{eq:star-skew-product}, $[a_{2i}X^{2i},bX]=(a_{2i}b-ba_{2i}^*)X^{2i+1}$ and $[a_{2i+1}X^{2i+1},b^*X]=(ba_{2i+1}-a_{2i+1}^*b^*)X^{2i+2}$. By letting $b=1$, we thus have $a_{2i}^*=a_{2i}$ and $a_{2i+1}^*=a_{2i+1}$. By then comparing coefficients in the original expression, $ba_{2i+1}=a^*_{2i+1}b^*=a_{2i+1}b^*$, and since $a_{2i+1}b^*=a_{2i+1}b$, we have $a_{2i+1}\in C(A)$. To show that the conditions are sufficient, it suffices to show that $[a_iX^i,bX^j]=0$ for any $i,j\in\mathbb{N}$. From \eqref{eq:star-skew-product}, $[a_iX^i,bX^j]=(\tau_j(a_i,*^i(b))-\tau_i(b,*^j(a_i)))X^{i+j}=(a_i*^i(b)-\tau_i(b,a_i))X^{i+j}=(a_i*^i(b)-a_ib)X^{i+j}=(a_ib-a_ib)X^{i+j}=0$. 

For a different proof, one can use \autoref{thm:cayley-dickson-isomorphism1} together with \cite[Theorem 6.8 (viii)]{McC85}.\\

\noindent\ref{it:Cayley-left-right-nucleus}--\ref{it:Cayley-middle-nucleus}: We note that $\sum_{i\in\mathbb{N}}a_iX^i\in N_x(B)$ where $a_i\in A$ and $x$ is any of $l,m,$ and $r$, if and only if $a_iX^i\in N_x(B)$ for any $i\in\mathbb{N}$. The necessary conditions in \ref{it:Cayley-left-right-nucleus}--\ref{it:Cayley-middle-nucleus} now follow by replacing $a$ and $bl$ in the proof of \cite[Theorem 6.8 (ix)--(x)]{McC85} by $a_{2i}X^{2i}$ and $a_{2i+1}X^{2i+1}$, respectively, and using that $X^{2i}\in Z(B)$ and $X^{2i+1}=XX^{2i}$ together with \autoref{thm:cayley-dickson-isomorphism}. (The proof regarding the right nucleus in \cite{McC85} has been omitted, however; it can be obtained by similar methods to that of the left nucleus.) To show that the conditions for $N_l(B)$ are sufficient, it suffices to show that $(a_iX^i,bX^j,cX^k)=0$ for any $b,c\in A$ and $i,j,k\in\mathbb{N}$. By \eqref{eq:star-skew-product}, $(a_iX^i,bX^j,cX^k)=(\tau_k(\tau_j(a_i,*^i(b)),*^{i+j}(c))-\tau_{j+k}(a_i,*^i(\tau_k(b,*^j(c)))))X^{i+j+k}=(\tau_k(a_i*^i(b),*^{i+j}(c))-a_i*^i(\tau_k(b,*^j(c))))X^{i+j+k}$, so if $k$ is even, then
\begin{align*}
(a_iX^i,bX^j,cX^k)&=((a_i*^i(b))*^{i+j}(c)-a_i*^i(b*^j(c)))X^{i+j+k}\\
&=((a_i*^i(b))*^{i+j}(c)-(a_i*^i(b))(*^{i+j}(c)))X^{i+j+k}=0.
\end{align*}
If $k$ is odd, then
\begin{align*}
(a_iX^i,bX^j,cX^k)&=(*^{i+j}(c)(a_i*^i(b))-a_i*^i(*^j(c)b))X^{i+j+k}\\
&=(*^{i+j}(c)(a_i*^i(b))-(a_i*^{i+j}(c))*^i(b))X^{i+j+k}\\
&=(*^{i+j}(c)(a_i*^i(b))-(*^{i+j}(c)a_i)*^i(b))X^{i+j+k}\\
&=(*^{i+j}(c)(a_i*^i(b))-*^{i+j}(c)(a_i*^i(b)))X^{i+j+k}=0.
\end{align*}
The conditions for $N_m(B)$ and $N_r(B)$ can be shown sufficient by similar calculations. Moreover, similarly to the previous proofs, one can also use \autoref{thm:cayley-dickson-isomorphism1} together with \cite[Theorem 6.8 (ix)--(x)]{McC85}.
\\

\noindent\ref{it:Cayley-nucleus}--\ref{it:Cayley-center}: This follows from the equalities $N(B)=N_l(B)\cap N_m(B)\cap N_r(B)$ and $Z(B)=C(B)\cap N(B)$ together with \ref{it:Cayley-commuter}--\ref{it:Cayley-middle-nucleus} above.
\end{proof}

With $Z_*(B)\colonequals \{b\in Z(B)\colon b^*=b\}$, we have the following corollary:

\begin{corollary}\label{cor:repeated-Cayley-Dickson}
If $B=A[X;*]^{\fl}$, $1/2\in K$, and $A$ is obtained from $K$ by $n$ repeated applications of the Cayley-Dickson process, then the following equalities hold:\newpage
\begin{enumerate}[label=\upshape(\roman*)]
\item $C(B)=Z(B)=\begin{cases}K[X]&\text{if }n=0,\\
			K[X^2]&\text{if } n\geq 1.
\end{cases}$
;\label{it:Cayley-extension-center}
\item $Z_*(B)=K[X^2]$;\label{it:Cayley-extension-starcenter}
\item $N(B)=\begin{cases}B&\text{if }0\leq n \leq 1,\\
			K[X^2]&\text{if } n\geq 2.
\end{cases}$\label{it:Cayley-extension-nucleus}
\end{enumerate}
\end{corollary}

\begin{proof}
\noindent\ref{it:Cayley-extension-center}: By \cite[Corollary 6.9 (i)]{McC85}, $C(A)=Z(A)$, and so by \ref{it:Cayley-commuter} and \ref{it:Cayley-center} in \autoref{thm:cayley-dickson-centers}, $C(B)=Z(B)$. Moreover, by \cite[Corollary 16.9 (ii)]{McC85}, $Z_*(A)=K$. If $n=0$, then $A=K$, and so $a^*=a$ for any $a\in A=K$. Hence $ka^*=ka$ for any $k\in K$, so by \ref{it:Cayley-center} in \autoref{thm:cayley-dickson-centers}, $Z(B)=K[X]$ if $n=0$. If $n\geq 1$, then $A=\Cay(C,\mu)$ for some $C$, so an arbitrary element $a\in A$ is of the form $(b,c)$ for some $b,c\in C$ and $(b,c)^*=(b^*,-c)$. Then $ka^*=ka$ for some $k\in K$ if and only if $kb^*=kb$ and $2kc=0$. Since $1/2\in K$, $c$ may be chosen arbitrarily, and since $K$ acts faithfully on $C$ by assumption, we have $2kc=0\implies k=0$. Hence, by \ref{it:Cayley-center} in \autoref{thm:cayley-dickson-centers}, $Z(B)=K[X^2]$ if $n\geq 1$.\\

\noindent\ref{it:Cayley-extension-starcenter}: By \ref{it:Cayley-extension-center}, $Z(B)=K[X]$ if $n=0$. If $p=k_0+k_1X+k_2X^2+\cdots\in K[X]=Z(B)$ for some $k_0,k_1,k_2\in K$, then $p^*=k_0-k_1X+k_2X^2+\cdots$. Hence $p^*=p$ if and only if $2k_{2i+1}=0$ for any $i\in\mathbb{N}$. Since $1/2\in K$, the latter equality is equivalent to $k_{2i+1}=0$ for any $i\in\mathbb{N}$. Therefore, $p\in Z_*(B)$ if and only if $p=k_0+k_2X^2+k_4X^4+\cdots$, so $Z_*(B)=K[X^2]$ if $n=0$. Moreover, by \ref{it:Cayley-extension-center}, $Z(B)=K[X^2]$ if $n\geq1$, so by a similar argument, $Z_*(B)=K[X^2]$ if $n\geq 1$.\\

\noindent\ref{it:Cayley-extension-nucleus}: If $0\leq n\leq 1$, then $A$ is associative and commutative. By \ref{it:cayley-dickson-properties3} in \autoref{thm:Cayley-Dickson-properties}, $B$ is associative, so $N(B)=B$. By \cite[Corollary 16.9 (iii)]{McC85}, $Z(A)=K$ if $n\geq 2$. Moreover, by the proof of \cite[Corollary 16.9 (iv)]{McC85}, $k[A,A]=0\implies k=0$ for any $k\in K$ if $n\geq 2$. Hence, by \ref{it:Cayley-extension-nucleus} in \autoref{thm:cayley-dickson-centers}, $N(B)=K[X^2]$ if $n\geq 2$.
\end{proof}	

The next example follows immediately from \autoref{cor:repeated-Cayley-Dickson} and \autoref{ex:classical-Cayley-Dickson}.

\begin{example}If $B=A[X;*]^{\fl}$ and $K=\mathbb{R}$, then
\begin{align*}
C(B)&=Z(B)=\begin{cases}\mathbb{R}[X]&\text{if } A=\mathbb{R},\\
 \mathbb{R}[X^2]&\text{if } A=\mathbb{C}, \mathbb{C}', \mathbb{H}, \mathbb{H}',\ldots,\end{cases}\\
 Z_*(B)&=\mathbb{R}[X^2]\text{ if } A=\mathbb{R},\mathbb{C},\mathbb{C}',\ldots,\\
 N(B)&=\begin{cases}B&\text{if } A=\mathbb{R}, \mathbb{C},\text{or } \mathbb{C}',\\
 \mathbb{R}[X^2]&\text{if } A=\mathbb{H}, \mathbb{H}',\mathbb{O}, \mathbb{O}',\ldots.\end{cases}
 \end{align*}
\end{example}

\section*{Acknowledgments}
We would like to thank the anonymous referee for comments on the manuscript which improved its presentation. We would also like to thank Johan Richter for pointing out a mistake in the formulation of an earlier version of \autoref{def:flipped-ring}.

\end{document}